\documentclass[10pt]{article}
\usepackage[a4paper, margin=0.75in]{geometry}
\usepackage[mathscr]{euscript}
\usepackage{amsmath,amsthm}
\usepackage{amsfonts,amssymb}
\usepackage{graphicx}
\usepackage{breqn}
\usepackage{mathrsfs}
\usepackage{amsmath}
\usepackage{float}
\usepackage{caption}
\usepackage{chngcntr}
\usepackage{tikz-qtree}
\usetikzlibrary{arrows.meta}
\usepackage{subcaption}
\usepackage{epsfig,amsmath,amsthm,latexsym,amssymb,amsgen,graphicx}
\usepackage{verbatim}
\usepackage{cite}

\newtheorem{thm}{Theorem}[section]

\newtheorem{cor}[thm]{Corollary}

\newtheorem{prop}[thm]{Proposition}

\newtheorem{defn}[thm]{Definition}
\newtheorem{rem}[thm]{Remark}

\newtheorem{exam}[thm]{Example}

\begin{document}
\title {\bf{Operations on Fuzzy Incidence Graphs and Strong Incidence Domination}}
\author { \small Kavya. R. Nair \footnote{Corresponding author. Kavya. R. Nair, Department of Mathematics, National Institute of Technology, Calicut, Kerala, India. Email: kavyarnair@gmail.com} and M. S. Sunitha \footnote{sunitha@nitc.ac.in} \\ \small Department of Mathematics, National Institute of Technology, Calicut, Kerala, India-673601}
\date{}
\maketitle

\hrule
\begin{abstract}
  Fuzzy incidence graphs (FIG) model real world problems efficiently when there is an extra attribute of vertex- edge relationship. The article discusses the operations on Fuzzy incidence graphs. The join, Cartesian product, tensor product, and composition of FIGs are explored. The study is concentrated mainly on strong fuzzy incidence graphs (SFIG). The idea of strong incidence domination (SID) is used, and strong incidence domination number (SIDN) in operations is examined. Basic properties of FIGs obtained from the operations are studied. Bounds for the domination number of product of two SFIGs are determined for the Cartesian and tensor products. Study is conducted on FIGs with strong join and composition. Complete fuzzy incidence graphs (CFIGs) and FIGs with effective pairs are also considered in the study.

\end{abstract}

\hrule
\section{Introduction}
With the introduction of fuzzy set theory by Zadeh \cite{zadeh} in 1965, the theory has evolved in various ways across numerous fields. The theory has wide range of applications in operation research, pattern recognition, decision theory, artificial intelligence etc. Graph-theoretical principles are frequently employed in the research and modeling of diverse applications in various fields. 
 However, in many circumstances, graph-theoretical notions are ambiguous and imprecise. In such cases it is suitable to use fuzzy set approaches to cope with ambiguity and uncertainty. This led to the introduction of fuzzy graph (FG) theory by Rosenfeld \cite{rosenfeld1975fuzzy} in 1975. The introduction of fuzzy incidence graphs (FIG) was motivated by a FG model with the extra attribute of vertices having some effect on the edges. Dinesh \cite{article} developed the term FIG in 2016 and studied some of its properties. The idea of connectivity and fuzzy end nodes is developed by Mordeson et al. \cite{book, mathew2019incidence}. 
It is often customary to study and perform operations on graph structures to obtain new structures from the existing ones. Since the 1950s, a number of graph products have been investigated. The operations on fuzzy graphs are discussed by Mordeson and Peng \cite{mordeson}. Parvathi et al. \cite{parvathy} examined several intuitionistic fuzzy graph (IFG) operations such as Cartesian product and composition. Sahoo and Pal \cite{sahoo} explored direct product, strong product and semistrong product in IFGs. Nazeer et al. studied the intuitionistic fuzzy incidence graphs (IFIG) products in \cite{nazeer2021}. The research on domination theory in graphs and fuzzy graphs expanded over the years due to its applications in numerous fields. The concept of domination emerged in the 1850s with the chessboard domination problem. Ore \cite{ore1962} and Berge \cite{berge1962} pioneered the study of domination in graphs in 1962. Significant works on fuzzy graph domination and domination in fuzzy graph products have been done in \cite{nagoor2006, manjusha2015strong, somasundaram1998domination, soma2, bhutani, ponnappan}. 
Nazeer et al. \cite{nazeer2021domination1, nazeer2021domination2} proposed FIG domination based on effective pairs and defined strong domination in FIG and their join. Afsharmanesh et al. \cite{afsharmanesh2021domination} recently proposed domination using valid edges in FIGs. Kavya and Sunitha \cite{kav} defined strong incidence domination using weight of strong pairs. 
\\
The study is motivated by the fact that the pairs in every FIG may be characterised as strong or non-strong. Hence the idea of SID developed can always be applied in any FIG. Furthermore, unlike other domination parameters that use the weight of vertices, SID uses the weight of strong pairs. As a result, the SID yields the lowest value. 
Also, graph operations are always beneficial for generating new structures from the already known structures. Hence the study on operations on FIG is conducted, and the idea of SID is applied to the obtained graphs.\\
The article studies some of the operations on FIGs. The study mainly deals with the join, Cartesian product, tensor product, and composition of SFIGs. Section 1 sums up the preliminaries. Section 2 discusses the join of FIGs and SID in the join. A characterisation for a FIG to be a SFIG is proved. It is illustrated with an example that the join of two SFIG need not be strong in general. A sufficient condition for the join to be strong is considered. Section 3 and 4 deal with the Cartesian and tensor products of SFIGs and FIGs with effective pairs respectively. Section 4 studies the composition of SFIGs. A sufficient condition for the composition to be strong is proved. Bound for the SIDN is obtained in each of the sections.

\section{Preliminaries}
The following definitions are taken from \cite{article, book, mathew2019incidence, kav, nazeer2021domination1, nazeer2021domination2, nazeer2021}.
 Throughout the article minimum and maximum operators are represented by $\wedge$ and $\vee$ respectively.
A triple $X= (V,E,I)$ such that $V$ is non-empty, $E\subseteq V\times V \text{ and } I\subseteq V\times E$ is an incidence graph (IG).\\
Elements in I are called incidence pairs or pairs and are of the form $(x,yz)$ where $x\in V$ and $e'=yz\in E$.\\
If $(w,xw), (x,xw), (y,yz) \text{ and } (z,yz)$ are in $I$, then the edges $wx$ and $yz$ are considered adjacent.\\
An incidence subgraph, $Y$ of $X$ is an IG having all its vertices, edges and pairs in $X$. 
\\
A sequence of vertices, edges and pairs starting at $x'$ and ending at $y'$, where $x',y'\in V\cup E$ is  called an incidence walk from $x'$ to $y'$. An incidence trail is an incidence walk with distinct pairs. If the vertices in an incidence walk are distinct , then it is called an incidence path.  
If each vertex in an IG is joined to every other vertex by a path, then the IG is connected. A component of an IG is a maximally connected incidence subgraph of the IG. \\
Let $X=(V,E)$ be a graph. Let $\varepsilon$ and $\rho$ be fuzzy subsets of $V$ and $E\subseteq V\times V$ respectively. Then $\mathscr X=(V,\varepsilon, \rho)$ is fuzzy graph(FG) of $X$ if $\rho(xy)\leq \varepsilon(x)\wedge\varepsilon(y)$ for all $x,y\in V$. Also, if $\eta(v',e')\leq \varepsilon(v')\wedge\rho(e')$, for all $v'\in V \text{ and } e'\in E$, then $\eta $ is the fuzzy incidence of $X$. And, $\mathscr{\tilde{X}}=(\varepsilon, \rho, \eta)$ is called fuzzy incidence graph (FIG) of $X$.\\
Here, $\varepsilon^*, \rho^* \text{ and } \eta^*$ are defined as $\varepsilon^*=\{x\in V: \varepsilon(x)>0\}$, $\rho^*=\{e'\in E: \rho(e')>0\}$, and $\eta^*=\{(x,xy)\in I: \eta(x,xy)>0\}$. If $|\varepsilon^*|=1$, then the FIG is called trivial.\\
Also, $\mathscr X=(V,\varepsilon,\rho)$ and $\mathscr X^*=(V,\varepsilon^*,\rho^*)$ are the underlying FG of $\mathscr{\tilde{X}}$ and underlying graph of $\mathscr X$ respectively.
Let $xy\in \rho^*$, if $(x,xy),(y,xy)\in \eta^* $, then $xy$ is an edge in $\tilde{\mathscr{X}}$. Pairs in FIG, $\mathscr{\tilde{X}}$ are elements of the form $(x,xy)$. If vertices $x$ and $y$ are joined by a path, then $x$ and $y$ are connected in $\tilde{\mathscr{X}}$. If each vertex is joined to every other vertex by 
a path, then $\tilde{\mathscr{X}}$ is connected. A fuzzy incidence subgraph 
$\tilde{\mathscr{Y}}=(\varphi,\xi,\zeta)$ of  $\tilde{\mathscr{X}}$ is such that 
$\varphi\subseteq \varepsilon, \xi\subseteq \rho, \text{ and } \zeta\subseteq \eta$ and  $\tilde{\mathscr{Y}}$ is a fuzzy incidence spanning subgraph of $\mathscr{\tilde{X}}$ if $\varphi= \varepsilon$. If $\varphi= \varepsilon, \xi= \rho, \text{ and } \zeta= \eta$ for elements in $\varphi^*, \xi^*, \zeta^*$ respectively, then $\tilde{\mathscr{Y}}$ is a subgraph of $\mathscr{\tilde{X}}$.\\
A FIG, $\tilde{\mathscr{X}}$ is complete fuzzy incidence graph (CFIG) if $\eta(x,xy)=\, \wedge \{\varepsilon(x), \rho(xy)\}$ for all $(x,xy)\in V\times E$ and $\rho(xy)=\varepsilon(x)\wedge\varepsilon(y) $ for all $(x,y)\in V\times V$. A pair $(x,xy)$ is an effective pair if $\eta(x,xy)=\, \wedge \{\varepsilon(x), \rho(xy)\}$. \\
In a FIG $\tilde{\mathscr{X}}$, a path from $s'$ to $t'$ where $s',t'\in \varepsilon^*\cup \rho^*$, is called an incidence path. The minimum of the $\eta $ values of pairs in an incidence path is the incidence strength of that path. Here, $\eta^{\infty}(x,yz) \text{ or } ICONN_{\tilde{\mathscr{X}}}(x,yz) $ is denoted as the incidence strength of path from $x$ to $yz$ of greatest incidence strength. \\
If $\tilde{\mathscr{X^*}}=(\varepsilon^*, \rho^*, \eta^*)$ is a cycle, then $\tilde{\mathscr{X}}$ is a cycle. In addition, $\tilde{\mathscr{X}}$ is a fuzzy cycle (FC), if there exists no unique edge $xy\in \rho^*$ with the least weight. A FIG, $\tilde{\mathscr{X}}=(\varepsilon, \rho, \eta)$ is a fuzzy incidence cycle (FIC) if $\tilde{\mathscr{X}}$ is a FC and there is no unique $(x,xy)\in \eta^*$ with the least weight. \\
Let $\tilde{\mathscr{X}}=(\varepsilon, \rho, \eta)$ be a FIG, $\eta^{\prime \infty}(x,yz)$ is the greatest incidence strength among the incidence strength of all paths from $x$ to $yz$ in $\tilde{\mathscr{G}}\setminus (x,yz)$. If $\eta(x,xy)>\eta^{\prime \infty}(x,xy)$, then the pair $(x,xy)$ is $\alpha-$ strong. Pair $(x,xy)$ is $\beta-$ strong if $\eta(x,xy)=\eta^{\prime \infty}(x,xy)$, and is a $\delta-$ pair if $\eta(x,xy)<\eta^{\prime \infty}(x,xy)$. A strong pair is $\alpha-$ strong or $\beta-$ strong pair. Vertex $x$ and edge $xy$ are strong fuzzy incidence neighbors if $(x,xy)$ is strong. A path with strong pairs is called strong incidence path (SIP). A FIG with only strong pairs is called a strong fuzzy incidence graph (SFIG). \\
If both $(x,xy)$ and $(y,xy)$ are strong, then $x$ is called strong incidence neighbor (SIN) of $y$. Also, $N_{IS}(x)$ is the strong incidence neighborhood of $x$, which is the set of all SINs of $x$. If $x=y$ or $x$ is a SIN of $y$, then $x$ dominates $y$. Isolated vertex $x$ is such that $N_{IS}(x)=\phi$. A set $\tilde{\mathscr D}\subseteq V$ in $\tilde{\mathscr{X}}$ is a strong incidence dominating set (SIDS), if for any $x\in V-\tilde{\mathscr D}$, $\exists$ some $y\in \tilde{\mathscr D}$ such that, $x$ is a SIN of $y$. \\
Here, $ W(\tilde{\mathscr D})$ is the weight of SIDS, $\tilde{\mathscr D}$,  defined as 
$$ W(\tilde{\mathscr D})=\sum_{x\in \tilde{\mathscr D}} \eta(x,xy) $$\\
where $\eta(x,xy)$ is minimum weight of strong pairs at $x$ and $y\in N_{IS}(x)$. The strong incidence domination number (SIDN), denoted as $\gamma_{IS}(\tilde{\mathscr X})$ or $\gamma_{IS}$ is the minimum weight of the SIDSs in the FIG, $\tilde{\mathscr{X}}$. A minimum SIDS is a SIDS with minimum weight.
\\
\section{Strong incidence domination in Join of Fuzzy Incidence Graph}
This section studies the SID in join of FIGs. The section begins with the definition of a WFIC. Theorem \ref{4} is a characterisation for a FIG to be strong. It is established that the join of two SFIG need not be strong. Proposition \ref{8} proves a necessary condition for the join of two FIGs to be strong. A sufficient condition for the join of two FIGs to be strong is also proved in Theorem \ref{10}. Theorem \ref{13} and its corollaries discuss the SIDN and SIDS in the join of FIGs.

\begin{rem}\label{1}
The definition of weak fuzzy incidence cycle (WFIC) is given in Definition \ref{2}. A WFIC is different from a FIC, since it is not necessary for the underlying FG of a WFIC to be a FC. An example of a WFIC is illustrated in Example \ref{3}.
\end{rem}
\begin{defn}\label{2}
Let $\tilde{\mathscr X}=(\varepsilon,\rho,\eta)$ be a FIG such that $(\varepsilon^*,\rho^*,\eta^*)$ is a cycle and there is no unique $(a,ab)\in \eta^*$ such that $\eta(a,ab)=\wedge\{\eta(c,cd)| (c,cd)\in \eta^*\}$. Then $\tilde{\mathscr X}$ is called a weak fuzzy incidence cycle (WFIC).
\end{defn}
\begin{center}
\begin{tikzpicture}[scale=0.5,inner sep=2.5pt]
\draw (0,0) node(1) [circle,fill,draw] {}
      (0,4) node(2) [circle,fill,draw] {}
      (5,0) node(3) [circle,fill,draw] {}
      (5,4) node(4) [circle,fill,draw] {}
      (2.5,6) node(5) [circle,fill,draw] {}
      ;
\draw[-] (1) to (2) to (5) to (4) to (3) to (1);
\draw (-0.6,4.1) node[above] {$y(1)$}
      (5.4,4.1) node[above] {$v(1)$}
      (5.2,-0.1) node[below] {$w(1)$}
      (-0.5,-0.1) node[below] {$x(1)$}
      (2.5,7) node[below] {$u(1)$}
      ;
\draw (-0.3,0.8) node[rotate=90] {\scriptsize{0.3}};
\draw (5.3,0.8) node[rotate=-90] {\scriptsize {0.2}};
\draw (0.8,-0.1) node[below] {\scriptsize{0.3}};
\draw (0.5,4.7) node[rotate=45] {\scriptsize{0.5}};
\draw (4.5,4.7) node[rotate=-45] {\scriptsize{0.1}};

\draw [->,dotted] (2.4,-0.2) -- (3.8,-0.2);
\draw [->,dotted] (3.8,0.2) -- (2.4,0.2);
\draw [->,dotted] (5.3,1.9) -- (5.3,3);
\draw [->,dotted] (4.8,3) -- (4.8,1.9);
\draw [->,dotted] (-0.3,1.9) -- (-0.3,3);
\draw [->,dotted] (0.3,3) -- (0.3,1.9);
\draw [->,dotted] (1,5.1) -- (1.9,5.8);
\draw [->,dotted] (1.9,5.2) -- (1,4.5);
\draw [->,dotted] (3.2,5.8) -- (4,5.1);
\draw [->,dotted] (4,4.5) -- (3.2,5.1);

\draw (3,-0.5) node[] {\scriptsize{0.3}};
\draw (3,0.5) node[] {\scriptsize{0.3}};
\draw (5.6,2.5) node[rotate=-90] {\scriptsize{0.2}};
\draw (4.5,2.5) node[rotate=90] {\scriptsize{0.2}};
\draw (-0.6,2.5) node[rotate=90] {\scriptsize{0.3}};
\draw (0.6,2.5) node[rotate=-90] {\scriptsize{0.3}};
\draw (1.4,5.8) node[rotate=45] {\scriptsize{0.5}};
\draw (1.5,4.5) node[rotate=45] {\scriptsize{0.5}};
\draw (3.6,5.8) node[rotate=-45] {\scriptsize{0.1}};
\draw (3.6,4.5) node[rotate=-45] {\scriptsize{0.1}};
\end{tikzpicture}
\\
Fig. 1. Weak fuzzy incidence cycle  $\tilde{\mathscr{X}}$
\end{center}
\begin{exam}\label{3}
The FIG, $\tilde{\mathscr{X}}$ in Fig. 1. is an example of WFIC, but $\tilde{\mathscr{X}}$ is not a FIC, since the underlying FG is not a FC.   
\end{exam}
\begin{thm}\label{4}
A FIG, $\tilde{\mathscr X}$ is a SFIG iff every cycle in $\tilde{\mathscr X}$ is a WFIC.
\end{thm}
\begin{proof}
Let $\tilde{\mathscr X}$ be a FIG such that every cycle in $\tilde{\mathscr X}$ is a WFIC.
Let $(a,ab)$ be a pair in $\tilde{\mathscr X}$.
Then there are 3 cases:\\
\textbf{Case 1}: $(a,ab)$ does not belongs to any cycle of $\tilde{\mathscr X}$.\\
Then $(a,ab)$ is a strong pair.\\
\textbf{Case 2}: In every cycle that contains $(a,ab)$, $(a,ab)$ is a pair with least weight.\\
Since every cycle in $\tilde{\mathscr X}$ is a WFIC this implies that $\eta(a,ab)=ICONN_{\tilde{\mathscr X}\setminus (a,ab)}(a,ab)$, i.e, $(a,ab)$ is a strong pair.\\
\textbf{Case 3}: There exists at least one cycle $\tilde{C}$ containing $(a,ab)$, in which $(a,ab)$ is not a pair with least weight.\\
Then $\eta(a,ab)$ is greater than the incidence strength of the path $\tilde{C}\setminus (a,ab)$, which implies that
$\eta(a,ab)\geq ICONN_{\tilde{\mathscr X}\setminus (a,ab)}(a,ab)$ i.e, $(a,ab)$ is a strong pair.\\
Since $(a,ab)$ is arbitrary, it implies that $\tilde{\mathscr X}$ is a SFIG.\\
Conversely, suppose that $\tilde{\mathscr X}$ is a SFIG.
Suppose that $\tilde{C}$ is a cycle that is not WFIC.
Then there exists a pair $(a,ab)$ in $\tilde{C}$ such that $(a,ab)$ is the unique weakest pair in $\tilde{C}$. The path $\tilde{C}\setminus (a,ab)$ has incidence strength greater than $\eta(a,ab)$. Hence 
$\eta(a,ab)<ICONN_{\tilde{\mathscr X}\setminus (a,ab)}(a,ab)$, i.e, $(a,ab)$ is a $\delta-$ pair, which contradicts the assumption. 
Hence the result.
\end{proof}
The definition of join of FIGs is taken from \textnormal{\cite{nazeer2021domination2}}.
\begin{defn}\label{5}\textnormal{\cite{nazeer2021domination2}}
Let $\tilde{\mathscr X_1}=(\varepsilon_1,\rho_1,\eta_1)$ and $\tilde{\mathscr X_2}=(\varepsilon_2,\rho_2,\eta_2)$ be two FIGs. Then the join of $\tilde{\mathscr X_1}$ and $\tilde{\mathscr X_2}$ denoted as $\tilde{\mathscr X_1}\oplus \tilde{\mathscr X_2}$ is the FIG, $\tilde{\mathscr X}=(\varepsilon,\rho,\eta)$ such that:
\scriptsize{
$$\varepsilon(a)=
\begin{cases}
\varepsilon_1(a)  & if\quad a\in \tilde{\mathscr X_1}\\
\varepsilon_2(a)  & if\quad a\in \tilde{\mathscr X_2}
\end{cases}
$$
$$\rho(ab)=
\begin{cases}
\rho_1(ab)  & if\quad a,b\in \tilde{\mathscr X_1}\\
\rho_2(ab)  & if\quad a,b\in \tilde{\mathscr X_2}\\
\varepsilon_1(a)\wedge \varepsilon_2(b)  & if\quad a\in \tilde{\mathscr X_1} \text{ and } b\in \tilde{\mathscr X_2}
\end{cases}
$$

$$\eta(a,ab)=
\begin{cases}
\eta_1(a,ab)  &if\quad (a,ab)\in \tilde{\mathscr X_1}\\
\eta_2(a,ab)  &if\quad (a,ab)\in \tilde{\mathscr X_2}\\
\varepsilon_1(a)\wedge \varepsilon_2(b)\wedge \eta_1(a,av_i)  &if \quad a\in \tilde{\mathscr X_1} \text{ and } b\in \tilde{\mathscr X_2}\\ &\qquad \text{ where } v_i\in \tilde{\mathscr X_1}\\
\varepsilon_1(b)\wedge \varepsilon_2(a)\wedge \eta_2(a,av_i)  &if \quad b\in \tilde{\mathscr X_1} \text{ and } a\in \tilde{\mathscr X_2} \\ &\qquad \text{ where } v_i\in \tilde{\mathscr X_2}
\end{cases}
$$
}
\end{defn}
\begin{rem}\label{6}
In general, the join of two SFIGs need not be strong. Example \ref{7} illustrates that join of two SFIGs need not be SFIG.
\end{rem}
\begin{figure}
    \centering
    \includegraphics[width=6cm, height=10cm]{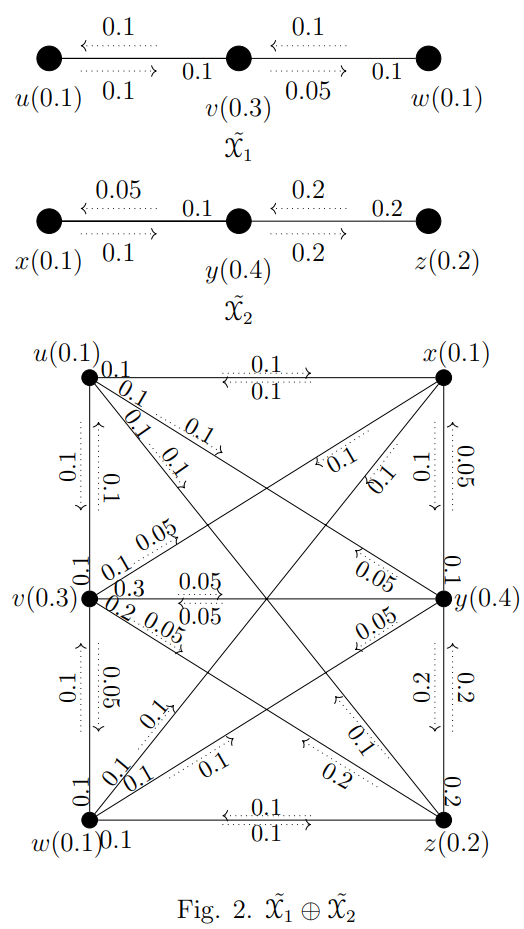}
\end{figure}

\begin{exam}\label{7}
For the FIGs in Fig. 2. $\tilde{\mathscr X_1}$ and $\tilde{\mathscr X_2}$ are SFIGs. But the join $\tilde{\mathscr X_1}\oplus \tilde{\mathscr X_2}$ is not a SFIG. The pair $(y,xy)$ is $\delta-$ pair since $ICONN_{\tilde{\mathscr X}-(y,xy)}=0.1>\eta(y,xy)=0.05$
Also, by Theorem \ref{4}, there exists a cycle $xyzux$ with unique weakest pair. Hence $\tilde{\mathscr X_1}\oplus \tilde{\mathscr X_2}$ is not strong.
\end{exam}
Next, Proposition \ref{8} gives a necessary condition for the join of two FIGs to be strong.

\begin{prop}\label{8}
If $\tilde{\mathscr X_1}$ and $\tilde{\mathscr X_2}$ are two FIGs such that the join $\tilde{\mathscr X}=\tilde{\mathscr X_1} \oplus \tilde{\mathscr X_2}$ is strong, then $\tilde{\mathscr X_1}$ and $\tilde{\mathscr X_2}$ are SFIGs.
\end{prop}
\begin{proof}
Let $\tilde{\mathscr X_1}$ and $\tilde{\mathscr X_2}$ be two FIGs such that the join $\tilde{\mathscr X}=\tilde{\mathscr X_1} \oplus \tilde{\mathscr X_2}$ is strong. Suppose that $\tilde{\mathscr X_1}$ is not strong. Then there exists a pair $(a,ab)$ such that $\eta(a,ab)< ICONN_{\tilde{\mathscr X_1}\setminus (a,ab)}(a,ab)$, i.e,  $\eta(a,ab)$ is less than the incidence strength of a path say, $P$ in $\tilde{\mathscr X_1}\setminus (a,ab)$. Then the path $P$ exists in the join $\tilde{\mathscr X}$ also. Since the incidence strength between $a$ and $ab$ is the maximum of incidence strength of all paths between $a$ and $ab$, $\eta(a,ab)< ICONN_{\tilde{\mathscr X}\setminus (a,ab)}(a,ab)$. 
Hence $(a,ab)$ is not strong in join which is a contradiction. Therefore, $\tilde{\mathscr X_1}$ is SFIG. By the same argument it can be proved that $\tilde{\mathscr X_2}$ is a SFIG. 
\end{proof}
\begin{prop}\label{9}
Let $\tilde{\mathscr X}$ be a FIG such that for each vertex $x$ in $\tilde{\mathscr X}$, every pair incident at $x$ has equal weight, then $\tilde{\mathscr X}$ is a SFIG.
\end{prop}
\begin{proof}
Suppose $\tilde{\mathscr X}$ is a FIG such that for each vertex $x$ in $\tilde{\mathscr X}$, every pair incident at $x$ has equal weight. Let $(a,ab)$ be an arbitrary pair in $\tilde{\mathscr X}$. Now, every FIP in $\tilde{\mathscr X}\setminus (a,ab)$ from $a$ to $ab$ begins with a pair incident at $a$. Since every pair incident at $a$ has equal weight, it implies that $\eta(a,ab)=\eta(a,ac)\geq ICONN_{\tilde{\mathscr X}\setminus (a,ab)}(a,ab)$.
Hence $(a,ab)$ is a strong pair. Since $(a,ab)$ is arbitrary, $\tilde{\mathscr X}$ is a SFIG.
\end{proof}

Theorem \ref{10} gives a sufficient condition for the join of two FIGs, $\tilde{\mathscr X_1}$ and $\tilde{\mathscr X_2}$ to be strong.

\begin{thm}\label{10}
Let $\tilde{\mathscr X_1}$ and $\tilde{\mathscr X_2}$ be FIGs such that for each vertex $u$ in $\tilde{\mathscr X_i}$, $i=1,2$, every pair incident at $u$ has the same weight. Then the join of $\tilde{\mathscr X_1}$ and $\tilde{\mathscr X_2}$ is a SFIG.
\end{thm}
\begin{proof}

Let $\tilde{\mathscr X}=\tilde{\mathscr X_1} \oplus \tilde{\mathscr X_2}$, and $a$ be a vertex in $\tilde{\mathscr X_1}$. Then, every pair in $\tilde{\mathscr X_1}$ incident at $a$ has equal weight say w. Now, for any vertex $b\in \tilde{\mathscr X_2}$ the pair $(a,ab)$ has weight 
\begin{equation*}
\begin{split}
\eta(a,ab) &= \varepsilon_1(a)\wedge \varepsilon_2(b)\wedge \eta(a,at_i),       t_i\in \tilde{\mathscr X_1}\\
&=\varepsilon_2(b)\wedge \text{w}
\end{split}
\end{equation*}
Therefore, every pair incident at $a$ in $\tilde{\mathscr X}$ will have weight less than or equal to w.\\
Hence for any pair $(a,ac)$, $a\in$ $\tilde{\mathscr X_1}$ there are two cases;\\
\textbf{Case 1}: $\eta(a,ac)=$ w.\\
Any incidence path from $a$ to $ac$ in $\tilde{\mathscr X}\setminus (a,ac)$ starts with pair incident at $a$. Since every pair at $a$ has weight either w or less than w, it follows that
$$ ICONN_{\tilde{\mathscr X}\setminus (a,ac)}(a,ac)\leq \text{w} = \eta(a,ac)$$
$\implies$ $(a,ac)$ is strong pair.\\
\textbf{Case 2}: $\eta(a,ac)=\varepsilon_2(c)$\\
By the definition of join of FIGs,  $\eta(c,ac)\leq \varepsilon_2(c)$.
Therefore as in Case 1, any incidence path from $a$ to $ac$ in $\tilde{\mathscr X}\setminus (a,ac)$ ends with pair $(c,ac)$. Hence, it follows that,
$$ ICONN_{\tilde{\mathscr X}\setminus (a,ac)}(a,ac)\leq \eta(c,ac) \leq \varepsilon_2(c)= \eta(a,ac)$$
$\implies$ $(a,ac)$ is strong pair.\\
Similarly it can be proved for the pairs incident at vertices in $\tilde{\mathscr X_2}$. Hence it implies that $\tilde{\mathscr X} $ is SFIG.
\end{proof}
\begin{rem}\label{11}
If two SFIGs have property as in Theorem \ref{10}, then their join is also a SFIG. 
The converse of Theorem \ref{10} need not be true as illustrated in Example \ref{12}.
\end{rem}
\begin{figure}[H]
    \centering
    \includegraphics[width=6cm,height=10cm]{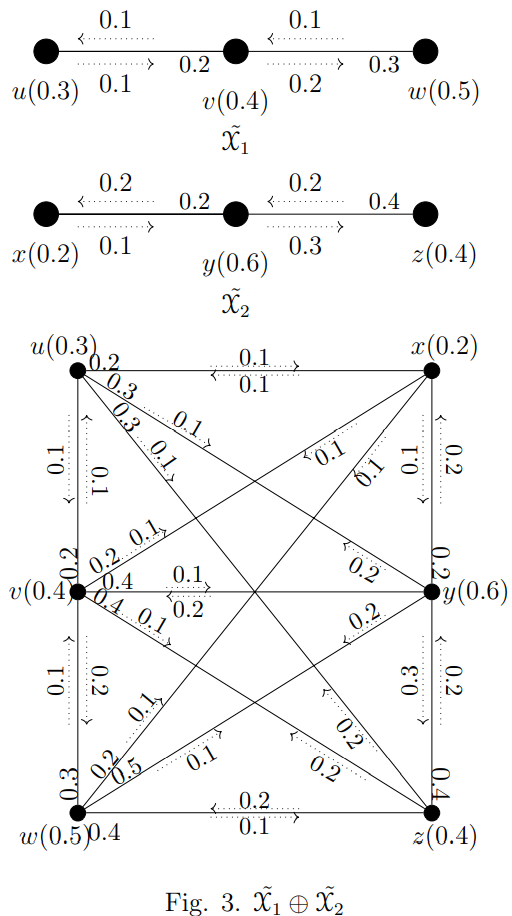}
\end{figure}
\begin{exam}\label{12}
The FIGs, $\tilde{\mathscr X_1}$, $\tilde{\mathscr X_2}$ and the join $\tilde{\mathscr X_1}\oplus \tilde{\mathscr X_2}$ in Fig. 3. are strong, but the pairs incident at $v$ in $\tilde{\mathscr X_1}$ have weights 0.1 and 0.2 and hence the converse of Theorem \ref{10} need not be true.
\end{exam}

\begin{thm}\label{13}
If $\tilde{\mathscr X_1}$ and $\tilde{\mathscr X_2}$ are FIGs such that $\tilde{\mathscr X}=\tilde{\mathscr X_1}\oplus \tilde{\mathscr X_2}$ is strong, then $$\gamma_{IS}= \wedge \{W(D_1),W(D_2),W(\{a,b\}) \text{ where } a\in \tilde{\mathscr X_1} \text{ and } b\in \tilde{\mathscr X_2}\}$$ where, $D_1$ and $D_2$ are minimum SIDS of $\tilde{\mathscr X_1}$ and $\tilde{\mathscr X_2}$ respectively. 
\end{thm}
\begin{proof}
Since $\tilde{\mathscr X}$ is strong, every pair in $\tilde{\mathscr X}$ is strong pair. Hence each vertex in $\tilde{\mathscr X_1}$ dominates every vertex in $\tilde{\mathscr X_2}$ and vice versa. Therefore $\{a,b\} $ where $a\in \tilde{\mathscr X_1}$ and $b\in \tilde{\mathscr X_2}$ dominates every vertex in $\tilde{\mathscr X}$, and hence is a SIDS. In a similar way if $D_1$ is a minimum SIDS of $\tilde{\mathscr X_1}$, it dominates every vertex of $\tilde{\mathscr X_1}$ and any vertex in $D_1$ dominates every vertex in $\tilde{\mathscr X_2}$ in the join. Hence $D_1$ is a SIDS of $\tilde{\mathscr X}$. By the same argument, $D_2$ is also a SIDS. Hence the result. 
\end{proof}
\begin{cor}\label{14}
Let $\tilde{\mathscr X_1}$ and $\tilde{\mathscr X_2}$ are FIGs such that $\tilde{\mathscr X}=\tilde{\mathscr X_1}\oplus \tilde{\mathscr X_2}$ is strong. If $D$ is the minimum SIDS of $\tilde{\mathscr X}$  then $|D|= min\{|D_1|,|D_2|,2\}$ where, $D_1$ and $D_2$ are minimum SIDS of $\tilde{\mathscr X_1}$ and $\tilde{\mathscr X_2}$ respectively.
\end{cor}
\begin{cor}\label{15}
Let $\tilde{\mathscr X_1}$ and $\tilde{\mathscr X_2}$ be FIGs such that for each vertex $u$ in $\tilde{\mathscr X_i}$, $i=1,2$, every pair incident at $u$ has the same weight. Then, $\gamma_{IS}(\tilde{\mathscr X_1}\oplus \tilde{\mathscr X_2})= min\{W(D_1),W(D_2),W(\{a,b\}) \text{ where } a\in \tilde{\mathscr X_1} \text{ and } b\in \tilde{\mathscr X_2}\}$ where, $D_1$ and $D_2$ are minimum SIDS of $\tilde{\mathscr X_1}$ and $\tilde{\mathscr X_2}$ respectively.
\end{cor}

\section{Strong incidence domination in Cartesian Product of Fuzzy Incidence Graph}
The section discusses the Cartesian product of FIGs, SFIGs and FIGs with effective pairs. Theorem \ref{19} proves that the Cartesian product of two SFIGs is strong. Proposition \ref{21} discusses the Cartesian product in FIGs with effective pairs. Proposition \ref{27} establishes the Cartesian product and the SID in CFIGs. The SID in the Cartesian product of SFIGs is also discussed in Theorem \ref{28}.
\begin{rem}\label{16}
The operations on IFIGs is defined and studied in \cite{nazeer2021}. In this article definition \ref{17}, \ref{29} and \ref{35} are modified from the definitions in \cite{nazeer2021}.
\end{rem}

\begin{defn}\label{17}
Let $\tilde{\mathscr X_1}=(V_1,E_1,I_1,\varepsilon_1,\rho_1,\eta_1)$ and $\tilde{\mathscr X_2}=(V_2,E_2,I_2,\varepsilon_2,\rho_2,\eta_2)$ be two FIGs. Then the Cartesian product of $\tilde{\mathscr X_1}$ and $\tilde{\mathscr X_2}$ denoted as $\tilde{\mathscr X_1}\times \tilde{\mathscr X_2}$ is the FIG, $\tilde{\mathscr X}=(V,E,I,\varepsilon,\rho,\eta)$ such that: $V=V_1 \times V_2$, $E=\{(a_1,b_1)(a_2,b_2)| a_1=a_2,\, b_1b_2 \in E_2 \text{ \, or \, } a_1a_2\in E_1,\,b_1=b_2\}$ and $I=\{((a_1,b_1), (a_1,b_1)(a_2,b_2))| a_1=a_2, (b_1,b_1b_2)\in I_2, (b_2,b_1b_2)\in I_2 \text{ \,or\, } b_1=b_2, (a_1,a_1a_2)\in I_1, (a_2,a_1a_2)\in I_1\}$ and 
\scriptsize{
$$\varepsilon(a_1, a_2)= \varepsilon_1(a_1)\wedge\varepsilon_2(a_2) \quad \forall (a_1,a_2)\in V_1\times V_2
$$
$$\rho((a_1,b_1)(a_2,b_2))=
\begin{cases}
\varepsilon_1(a_1)\wedge\rho_2(b_1b_2)  & if\quad a_1=a_2 , b_1b_2\in E_2\\
\rho_1(a_1a_2)\wedge\varepsilon_2(b_1)  & if\quad b_1=b_2 , a_1a_2\in E_1
\end{cases} 
$$
$$\eta((a_1,b_1), (a_1,b_1)(a_2,b_2))=
\begin{cases}
\varepsilon_1(a_1)\wedge\eta_2(b_1,b_1b_2)  & if\quad a_1=a_2 ,\\ & \qquad (b_1,b_1b_2)\in I_2\\
\eta_1(a_1,a_1a_2)\wedge\varepsilon_2(b_1)  & if\quad b_1=b_2 ,\\ & \qquad (a_1,a_1a_2)\in I_1
\end{cases}
$$
}
\end{defn}
\begin{rem}\label{18}
Let $\tilde{\mathscr X_1}$ and $\tilde{\mathscr X_2}$ be two CFIGs with $m$ and $n$ vertices respectively. Let $a_1,a_2,...,a_m$ be the vertices of $\tilde{\mathscr X_1}$ and $b_1,b_2,...,b_n$ be the vertices of $\tilde{\mathscr X_2}$.
Consider the vertices of $\tilde{\mathscr X_1}\times\tilde{\mathscr X_2}$ in the form of rows and columns as follows,\\
$(a_1,b_1), (a_1,b_2), (a_1,b_3),...,(a_1,b_n)$\\
$(a_2,b_1), (a_2,b_2), (a_2,b_3),...,(a_2,b_n)$\\
$\vdots$\\
$(a_m,b_1), (a_m,b_2), (a_m,b_3),...,(a_m,b_n)$\\
Each vertex is adjacent to all vertices in the same row and column in which it belongs. Since $\tilde{\mathscr X_1}$ and $\tilde{\mathscr X_2}$ are CFIGs, they are SFIGs. Hence, each vertex dominates all vertices in the same row and column in which it belongs.
\end{rem}
\begin{thm}\label{19}
Let $\tilde{\mathscr X_1}$ and $\tilde{\mathscr X_2}$ be two SFIGs, then $\tilde{\mathscr X}=\tilde{\mathscr X_1}\times\tilde{\mathscr X_2}$ is a SFIG.
\end{thm}
\begin{proof}
Let $\tilde{\mathscr X_1}$ and $\tilde{\mathscr X_2}$ be two SFIGs and $\tilde{\mathscr X}=\tilde{\mathscr X_1}\times\tilde{\mathscr X_2}$ be the Cartesian product of $\tilde{\mathscr X_1}$ and $\tilde{\mathscr X_2}$. Consider a pair say, $((a_1,b_1),(a_1,b_1)(a_2,b_2))$ in $\tilde{\mathscr X}$. Then, either $a_1=a_2, (b_1,b_1b_2), (b_2,b_1b_2)\in I_2$ or $(a_1,a_1a_2),(a_2,a_1a_2)\in I_1, b_1=b_2$.\\ 
Consider the case when $a_1=a_2, (b_1,b_1b_2),(b_2,b_1b_2)\in I_2$. Then, $\eta((a_1,b_1),(a_1,b_1)(a_2,b_2))= \varepsilon_1(a_1)\wedge\eta_2(b_1,b_1b_2)$.\\
\textbf{Case 1}: $\eta((a_1,b_1),(a_1,b_1)(a_2,b_2))= \varepsilon_1(a_1)$\\
Then, the pair weight
\begin{align*}
\eta((a_1,b_2),(a_1,b_2)(a_1,b_1))&= \varepsilon_1(a_1)\wedge\eta_2(b_2,b_1b_2)\\
&\leq \varepsilon_1(a_1)
\end{align*}
Any path from $(a_1,b_1)$ to $(a_1,b_1)(a_1,b_2)$ contains the pair $((a_1,b_2),(a_1,b_2)(a_1,b_1))$. Since the weight of \\$((a_1,b_2),(a_1,b_2)(a_1,b_1))$ is less than or equal to $\varepsilon_1(a_1)$, $((a_1,b_1),(a_1,b_1)(a_1,b_2))$ is a strong pair. \\
\textbf{Case 2}: $\eta((a_1,b_1),(a_1,b_1)(a_2,b_2))= \eta_2(b_1,b_1b_2)$. 
Now, consider paths from $(a_1,b_1)$ to $(a_1,b_2)$.\\
\textbf{Sub case 1}: Consider a path from $(a_1,b_1)$ to $(a_1,b_2)$ such that the vertices in the path are of the form $(a_1,v)$, $v\in V_2$. Corresponding to each such path, say, $P: (a_1,b_1), (a_1, v_1), (a_1, v_2),...,(a_1,v_n),(a_1,b_2)$, $v_1,v_2,...,v_n\in \tilde{\mathscr X_2}$ and $b_1v_1, v_1v_2,..., v_n b_2\in E_2$, there exists a path $P': b_1,v_1,v_2,...,v_n,b_2$ in $\tilde{\mathscr X_2}$. The path $P'$ together with $\{(b_2,b_1b_2), b_1b_2,(b_1,b_1b_2)\}$ forms a cycle, say $C$. Since $\tilde{\mathscr X_2}$ is a SFIG every cycle contains more than one weakest pair. Therefore, in $C$, there exists a pair with weight $\leq \eta_2(b_1,b_1b_2)$. Since $\varepsilon_1(a_1)\wedge\eta_2(b_1,b_1b_2)=\eta(b_1,b_1b_2)$, and the weight of each pair in path $P$ is defined as $\varepsilon_1(a_1)\wedge \eta_2(b_1,b_1v_1)$, 
$\varepsilon_1(a_1)\wedge \eta_2(v_1,b_1v_1)$,
$\varepsilon_1(a_1)\wedge \eta_2(v_1,v_1v_2)$,
$\varepsilon_1(a_1)\wedge \eta_2(v_2,v_1v_2)$
,..., $\varepsilon_1(a_1)\wedge \eta_2(v_n,v_nb_2)$,
$\varepsilon_1(a_1)\wedge \eta_2(b_2,v_nb_2)$, it implies that the path $P\cup \{((a_1,b_2),(a_1,b_1)(a_1,b_2)), (a_1,b_1)(a_1,b_2)\}$ has strength $\leq \eta_2(b_1,b_1b_2)$. \\
\textbf{Sub case 2}: Consider a path from $(a_1,b_1)$ to $(a_1,b_2)$ that contains vertices of the form $(a_n, b_1)$ and $(a_n,b_2)$, $a_n\neq a_1$.\\
In each such path there will be an edge of the form $(a_{i},b_1)(a_{i},b_2)$ with pair weights $\eta((a_{i},b_1),(a_{i},b_1)(a_{i},b_2))=\varepsilon_1(a_{i})\wedge\eta_2(b_1,b_1b_2)$ and $\eta((a_{i},b_2),(a_{i},b_1)(a_{i},b_2))=\varepsilon_1(a_{i})\wedge\eta_2(b_2,b_1b_2)$. Hence, the path contains pair with weight $\leq \eta_2(b_1,b_1b_2).$ Therefore, the strength of the path is $\leq\eta(b_1,b_1b_2).$\\
\textbf{Sub case 3}: Consider a path from $(a_1,b_1)$ to $(a_1,b_2)$ that contains vertices of the form $(a_n,b_m)$, $b_m\neq b_1, b_2$. Then corresponding to each such path, say, $P$ in $\tilde{\mathscr X}$, from $(a_1,b_1) $ to $(a_1,b_2)$, there exists a path in $\tilde{\mathscr X_2}$ from $b_1$ to $b_2$ other than the edge $b_1b_2$. This path in $\tilde{\mathscr X_2}$ together with $\{(b_2,b_1b_2),b_1b_2,(b_1,b_1b_1)\}$ forms a cycle, say, $C$  in $\tilde{\mathscr X_2}$. Since $\tilde{\mathscr X_2}$ is a SFIG, $C$ has more than one weakest pair. And hence $P\cup \{ (a_1,b_1)(a_1,b_2), ((a_1,b_2),(a_1,b_1)(a_1,b_2))\}$ in $\tilde{\mathscr X}$ has a pair of weight $\leq \eta_2(b_1,b_1b_2)$. Hence the pair $((a_1,b_1),(a_1,b_1)(a_1,b_2))$ is strong.
The case when $b_1=b_2, (a_1,a_1a_2), (a_2,a_1a_2)\in I_1$ can be proved similarly by taking $\tilde{\mathscr X_1}$ as SFIG. Therefore, if $\tilde{\mathscr X_1}$ and $\tilde{\mathscr X_2}$ are strong, then $\tilde{\mathscr X}$ is a SFIG.
\end{proof}
\begin{rem}\label{20}
By the Definition \ref{17}, of Cartesian product of FIGs, $\tilde{\mathscr X_1}$ and $\tilde{\mathscr X_2}$, a pair of non- adjacent vertices always exists in $\tilde{\mathscr X}=\tilde{\mathscr X_1}\times \tilde{\mathscr X_2}$. Therefore, the Cartesian product of two FIGs can never be CFIG.  
\end{rem}
\begin{prop}\label{21}
Let $\tilde{\mathscr X_1}$ and $\tilde{\mathscr X_2}$ be two FIGs with effective pairs, then $\tilde{\mathscr X}=\tilde{\mathscr X_1}\times\tilde{\mathscr X_2}$ is a FIG with effective pairs. 
\end{prop}
\begin{proof}
Let $\tilde{\mathscr X}=\tilde{\mathscr X_1}\times\tilde{\mathscr X_2}$ be the Cartesian product of two FIGs $\tilde{\mathscr X_1}$ and $\tilde{\mathscr X_2}$ with effective pairs. Consider a pair $((a_1,b_1),(a_1,b_1)(a_2,b_2))$ in $\tilde{\mathscr X}$, and the pair is effective if $\eta((a_1,b_1),(a_1,b_1)(a_2,b_2))=\varepsilon(a_1,b_1)\wedge\rho((a_1,b_1)(a_2,b_2))$.\\
Since $\tilde{\mathscr X_1}$ and $\tilde{\mathscr X_2}$ are FIGs with effective pairs and by Definition \ref{17}, the weight of $((a_1,b_1),(a_1,b_1)(a_2,b_2))$ is
{\scriptsize
\begin{align*}
\eta((a_1,b_1),(&a_1,b_1)(a_2,b_2))\\
&=
\begin{cases}
    \varepsilon_1(a_1)\wedge\eta_2(b_1,b_1b_2) & if \quad a_1=a_2,\quad (b_1,b_1b_2)\in I_2\\
    \eta_1(a_1,a_1a_2)\wedge\varepsilon_2(b_1) & if 
    \quad b_1=b_2, \quad (a_1,a_1a_2)\in I_1
\end{cases}\\
&=
\begin{cases}
  \varepsilon_1(a_1)\wedge\varepsilon_2(b_1)\wedge\rho_2(b_1b_2) & if \quad a_1=a_2,\\ &\qquad (b_1,b_1b_2)\in I_2\\ 
  \varepsilon_1(a_1)\wedge \rho_1(a_1a_2)\wedge \varepsilon_2(b_1) & if \quad b_1=b_2, \\ &\qquad (a_1,a_1a_2)\in I_1
\end{cases}\\
&=
\begin{cases}
  \varepsilon_1(a_1)\wedge\varepsilon_2(b_1)\wedge\varepsilon_1(a_1)\wedge\rho_2(b_1b_2) & if \quad a_1=a_2,\\ &\qquad (b_1,b_1b_2)\in I_2\\
  \varepsilon_1(a_1)\wedge \varepsilon_2(b_1)\wedge \rho_1(a_1a_2) \wedge \varepsilon_2(b_1) & if \quad b_1=b_2,\\ & \qquad (a_1,a_1a_2)\in I_1
\end{cases}\\
&= \varepsilon(a_1,b_1)\wedge\rho((a_1,b_1)(a_2,b_2))
\end{align*}
}
which implies that the pair $((a_1,b_1),(a_1,b_1)(a_2,b_2))$ is effective. Hence, $\tilde{\mathscr X}$ is a FIG with effective pairs.
\end{proof}
\begin{thm}\label{22}\cite{book}
Let $\tilde{\mathscr X}=(\varepsilon,\rho,\eta)$ be a FIG. If $\eta(x,xy)=\wedge\{\varepsilon(x),\rho(xy)\}$, then pair $(x,xy)$ is strong.
\end{thm}
\begin{prop}\label{23}
Let $\tilde{\mathscr X_1}$ and $\tilde{\mathscr X_2}$ be two FIGs such that there exits an edge $a_1a_2\in E_1$ with pair weights $\eta(a_1,a_1a_2), \eta(a_2,a_1a_2)\geq\varepsilon_2(v)$ for some $v\in V_2$, then the pairs $((a_1,v),(a_1,v)(a_2,v))$ and \\ $((a_2,v),(a_1,v)(a_2,v))$ are effective pairs.
\end{prop}
\begin{proof}
Suppose $\tilde{\mathscr X_1}$ and $\tilde{\mathscr X_2}$ are two FIGs such that there exits an edge $a_1a_2\in E_1$ with pair weights \\$\eta(a_1,a_1a_2), \eta(a_2,a_1a_2)\geq\varepsilon_2(v)$ for some $v\in V_2$. Since $\rho_1(a_1a_2)\geq\eta(a_1,a_1a_2)$ and $\rho_1(a_1a_2)\geq\eta(a_2,a_1a_2)$, the weight of edge $(a_1,v)(a_2,v)$ in $E$ is $\rho((a_1,v)(a_2,v))=\rho_1(a_1a_2)\wedge\varepsilon_2(v)=\varepsilon_2(v)$.\\ 
Therefore, $\eta((a_1,v),(a_1,v)(a_2,v))=\eta_1(a_1,a_1a_2)\wedge\varepsilon_2(v)=\varepsilon_2(v)=\rho((a_1,v)(a_2,v))\wedge\varepsilon_2(v)$.
Similarly, \\$\eta((a_2,v),(a_1,v)(a_2,v))=\eta_1(a_2,a_1a_2)\wedge\varepsilon_2(v)=\varepsilon_2(v)=\rho((a_1,v)(a_2,v))\wedge\varepsilon_2(v)$. Hence, the pairs \\$((a_1,v),(a_1,v)(a_2,v))$ and $((a_2,v),(a_1,v)(a_2,v))$ are effective pairs. 
\end{proof}
Corollary \ref{24} follows from Theorem \ref{22} and Proposition \ref{23}.
\begin{cor}\label{24}
Let $\tilde{\mathscr X_1}$ and $\tilde{\mathscr X_2}$ be two FIGs such that there exits an edge $a_1a_2\in E_1$ with pair weights $\eta(a_1,a_1a_2), \eta(a_2,a_1a_2)\geq\varepsilon_2(v)$ for some $v\in V_2$, then the pairs $((a_1,v),(a_1,v)(a_2,v))$ and \\ $((a_2,v),(a_1,v)(a_2,v))$ are strong pairs.
\end{cor}
\begin{prop}\label{25}
Let $\tilde{\mathscr X_1}$ and $\tilde{\mathscr X_2}$ be two FIGs. If the pairs $(a_1,a_1a_2)$ and $(a_2,a_1a_2)$ are effective pairs in $\tilde{\mathscr X_1}$, then for any $v\in V_2$, the pairs $((a_1,v),(a_1,v)(a_2,v))$ and $((a_2,v),(a_1,v)(a_2,v))$ are effective in $\tilde{\mathscr X}$.
\end{prop}
\begin{proof}
Consider two FIGs, $\tilde{\mathscr X_1}$ and $\tilde{\mathscr X_2}$. Since $(a_1,a_1a_2)$ and $(a_2,a_1a_2)$ are effective pairs in $\tilde{\mathscr X_1}$, $\eta_1(a_1,a_1a_2)=\rho_1(a_1a_2)\wedge\varepsilon_1(a_1)=\rho_1(a_1a_2)$ and  $\eta_1(a_2,a_1a_2)=\rho_1(a_1a_2)\wedge\varepsilon_1(a_2)=\rho_1(a_1a_2)$.
Therefore, \\$\eta((a_1,v),(a_1,v)(a_2,v))=\eta_1(a_1,a_1a_2)\wedge\varepsilon_2(v)=\rho_1(a_1a_2)\wedge\varepsilon_2(v)=\rho((a_1,v)(a_2,v))\wedge\varepsilon_2(v)$ and \\ $\eta((a_2,v),(a_1,v)(a_2,v))=\eta_1(a_2,a_1a_2)\wedge\varepsilon_2(v)=\rho_1(a_1a_2)\wedge\varepsilon_2(v)=\rho((a_1,v)(a_2,v))\wedge\varepsilon_2(v)$. Hence the pairs $((a_1,v),(a_1,v)(a_2,v))$ and $((a_2,v),(a_1,v)(a_2,v))$ are effective in $\tilde{\mathscr X}$. 
\end{proof}
\begin{cor}\label{26}
Let $\tilde{\mathscr X_1}$ and $\tilde{\mathscr X_2}$ be two FIGs. If the pairs $(a_1,a_1a_2)$ and $(a_2,a_1a_2)$ are effective pairs in $\tilde{\mathscr X_1}$, then for any $v\in V_2$, the pairs $((a_1,v),(a_1,v)(a_2,v))$ and $((a_2,v),(a_1,v)(a_2,v))$ are strong in $\tilde{\mathscr X}$.
\end{cor}
\begin{prop}\label{27}
Let $\tilde{\mathscr X_1}$ and $\tilde{\mathscr X_2}$ be two CFIGs with $m$ and $n$ vertices respectively. Let $D$ be a minimum dominating set of $\tilde{\mathscr X_1}\times\tilde{\mathscr X_2}$, then $|D|=\wedge\{m,n\}$ and $\gamma_{IS}=\wedge\{m,n\}\times$w where w is the weight of vertex having the least weight in $V_1 \cup V_2$.
\end{prop}
\begin{proof}
Let $\tilde{\mathscr X_1}$ and $\tilde{\mathscr X_2}$ be two CFIGs with $m$ and $n$ vertices respectively. Let $a_1,a_2,...,a_m$ be the vertices of $\tilde{\mathscr X_1}$ and $b_1,b_2,...,b_n$ be the vertices of $\tilde{\mathscr X_2}$. Suppose that $m<n$. Then, $D_1=\{(a_1,b_1),(a_2,b_2),....,(a_m,b_m)\} $ will dominate all vertices of $\tilde{\mathscr X}$. This is because, $(a_1,b_1)$ dominates the vertices in the set $\{(a_1,b_2),(a_1,b_3),...,(a_1,b_n)\}$, $(a_2,b_2)$ dominates the vertices in $\{(a_2,b_1),(a_2,b_3),...,(a_2,b_n)\}$ and so on. Finally, $(a_m,b_m) $ dominates vertices in $\{(a_m,b_1),(a_m,b_2),...,(a_m,b_n)\}$. Therefore all vertices in $\tilde{\mathscr X}$ are dominated by $m$ vertices in  $D_1$. \\
\textbf{Case 1}: Vertex with least weight belongs to $\tilde{\mathscr X_1}.$\\
Without loss of generality assume that $a_1$ is a vertex in $\tilde{\mathscr X_1}$ with least weight. Then the pairs incident at $a_1$ has weight $\varepsilon_1(a_1)$. Now, consider the dominating set $D_1$. The first vertex $(a_1,b_1)$ is adjacent to $(a_1,b_2)$,  $\eta((a_1,b_1),(a_1,b_1)(a_1,b_2))=\varepsilon_1(a_1)\wedge\eta_2(b_1,b_1b_2)=\varepsilon_1(a_1)\wedge\varepsilon_2(b_1)\wedge\rho_2(b_1b_2)=\varepsilon_1(a_1)\wedge\varepsilon_2(b_1)\wedge\varepsilon_2(b_2)$=$\varepsilon_1(a_1)$ i.e, the minimum of the weight of the pairs at $(a_1,b_1)$ is $\varepsilon_1(a_1)$. Similarly, $(a_2,b_2)$ is adjacent to $(a_1,b_2)$ and the weight of the pair $((a_2,b_2),(a_2,b_2)(a_1,b_2))$ is $\varepsilon_1(a_1)$ and so on. In the same way $(a_m,b_m)$ is adjacent to $(a_1,b_m)$ and the minimum of the weight of the pairs at $(a_m,b_m)$ is $\varepsilon_1(a_1)$. Therefore, the weight of $D_1$ is $m\varepsilon_1(a_1)$.\\
\textbf{Case 2}: Vertex with least weight belongs to $\tilde{\mathscr X_2}.$\\
Without loss of generality assume that $b_1$ is a vertex in $\tilde{\mathscr X_2}$ with least weight. Then the pairs incident at $b_1$ has weight $\varepsilon_2(b_1)$. Now consider the dominating set $D_1$. The first vertex $(a_1,b_1)$ is adjacent to $(a_2,b_1)$, and the weight of the pair $((a_1,b_1),(a_1,b_1)(a_2,b_1))$ is $\varepsilon_2(b_1)$, i.e, the minimum of the weight of the pairs at $(a_1,b_1)$ is $\varepsilon_2(b_1)$. Similarly, $(a_2,b_2)$ is adjacent to $(a_2,b_1)$ and the weight of the pair $((a_2,b_2),(a_2,b_2)(a_2,b_1))$ is $\varepsilon_2(b_1)$ and so on. In the same way $(a_m,b_m)$ is adjacent to $(a_m,b_1)$ and the minimum of the weight of the pairs at $(a_m,b_m)$ is $\varepsilon_2(b_1)$. Therefore, the weight of $D_1$ is $m\varepsilon_2(b_1)$.\\
Among the dominating sets with $m$ vertices, the least weight is obtained for $D_1$ which is $m$w, where w is the weight of vertex in $\tilde{\mathscr X_1}$ or $\tilde{\mathscr X_2}$ having the least weight. 
Now, assume $D_2=\{a_1,a_2,...,a_{m-1}\}$ is a dominating set of $\tilde{\mathscr X_1}$ with $m-1$ vertices. Since there are $m$ rows of vertices in $\tilde{\mathscr X}$, there exists a row or column of vertices which is not dominated by any of the $m-1$ vertices. \\
Hence, the minimum dominating set contains $m$ vertices and $\gamma_{IS}=m$w where w is the vertex having the least weight in $V_1 \cup V_2$.\\
The case of $n<m$ and $n=m$ can be proved in a similar way and in that case the minimum dominating set contains $n$ vertices and $\gamma_{IS}=n$w where w is the weight of vertex having the least weight.\\
Hence the result.
\end{proof}
\begin{thm}\label{28}
Let $\tilde{\mathscr X_1}$ and $\tilde{\mathscr X_2}$ be two SFIGs, and $D_1$ and $D_2$ are  SIDS of $\tilde{\mathscr X_1}$ and $\tilde{\mathscr X_2}$ respectively. Then $D_1\times V2$ and $V_1\times D_2 $ are SIDS of $\tilde{\mathscr X}$ and $$\gamma_{IS}\leq \wedge\{W(D_1 \times V_2), W(V_1\times D_2)\}.$$
\end{thm}
\begin{proof}
Let $\tilde{\mathscr X_1}$ and $\tilde{\mathscr X_2}$ be two SFIGs, and $D_1$ and $D_2$ are  SIDS of $\tilde{\mathscr X_1}$ and $\tilde{\mathscr X_2}$ respectively. By Theorem \ref{19}, $\tilde{\mathscr X}$ is a SFIG.\\
Let $(u,v)$ be any vertex in $\tilde{\mathscr X}$. There are two cases:\\
\textbf{Case 1}: At least one of $u$ or $v$ or both are isolated.\\
\textbf{Sub case 1}: Vertex $u$ is isolated in $\tilde{\mathscr X_1}$.\\
Since $D_1$ is a SIDS, $u\in D_1$. If $v\in D_2$, then $(u,v)\in D_1 \times V_2$ and $(u,v)\in V_1\times D_2$. And if $v\notin D_2$, then there exists a vertex $v'\in D_2$ such that $v'$ dominates $v$ in $\tilde{\mathscr X_2}$. Hence, $(u,v')\in V_1\times D_2$ dominates $(u,v)$ in $\tilde{\mathscr X}$. Also, $(u,v)\in D_1 \times V_2$. \\
\textbf{Sub case 2}: Vertex $v$ is isolated in $\tilde{\mathscr X_2}$.\\
Since $D_2$ is a SIDS, $v\in D_2$. If $u\in D_1$, then $(u,v)\in D_1 \times V_2$ and $(u,v)\in V_1\times D_2$. And if $u\notin D_1$, then there exists a vertex $u'\in D_1$ such that $u'$ dominates $u$ in $\tilde{\mathscr X_1}$. Hence, $(u',v)\in D_1 \times V_2$ dominates $(u,v)$ in $\tilde{\mathscr X}$. Also, $(u,v)\in V_1\times D_2$.\\
\textbf{Sub case 3}: Both vertices $u$ and $v$ are isolated in $\tilde{\mathscr X_1}$ and $\tilde{\mathscr X_2}$ respectively. \\
Then $(u,v)\in D_1 \times V_2$ and $(u,v)\in V_1\times D_2$.\\
Hence if at least one of $u$ and $v$ is an isolated vertex in $\tilde{\mathscr X_1}$ and $\tilde{\mathscr X_2}$ respectively, then $(u,v)$ either belongs to $D_1 \times V_2$ and $V_1\times D_2$ or is dominated by vertex in $D_1 \times V_2$ and $V_1\times D_2$.\\
\textbf{Case 2}: Both $u$ and $v$ are non isolated.\\
Consider $D_1 \times V_2$, and if $(u,v)\notin D_1 \times V_2$, then $u\notin D_1$ and
since $D_1$ is a SIDS of $\tilde{\mathscr X_1}$, there exists $u'\in D_1$ such that $(u,uu')$ and $(u',uu')$ are strong pairs in $\tilde{\mathscr X_1}$, i.e., $u'$ dominates $u$ in $\tilde{\mathscr X_1}$. Then, by the definition of Cartesian product of FIGs $(u,v)$ and $(u',v)$ are adjacent.
Therefore, $(u,v)(u',v)\in E$ and $((u,v),(u,v)(u',v))$ and $((u',v),(u,v)(u',v))$ are strong pairs in $\tilde{\mathscr X}$, i.e., $(u',v)$ dominates $(u,v)$.
Hence $D_1 \times V_2$ is a SIDS in $\tilde{\mathscr X}$. Similarly, $V_1\times D_2$ is also a SIDS in $\tilde{\mathscr X}$.\\
Therefore, clearly $\gamma_{IS}\leq \wedge\{W(D_1 \times V_2), W(V_1\times D_2)\}$.
\end{proof}

\section{Strong Incidence Domination in Tensor Product of Fuzzy Incidence Graphs}
The section discusses the tensor product of FIGs, SFIGs and FIGs with effective pairs. Theorem \ref{30} proves that the tensor product of FIGs with effective pairs is a FIG with effective pairs. Theorem \ref{33} establishes that the tensor product of two SFIGs is a SFIG. Theorem \ref{34} discusses the SID in the tensor product of two SFIG.
\begin{defn}\label{29}
Let $\tilde{\mathscr X_1}=(V_1,E_1,I_1,\varepsilon_1,\rho_1,\eta_1)$ and $\tilde{\mathscr X_2}=(V_2,E_2,I_2,\varepsilon_2,\rho_2,\eta_2)$ be two FIGs. Then the tensor product of $\tilde{\mathscr X_1}$ and $\tilde{\mathscr X_2}$ denoted as $\tilde{\mathscr X_1}\diamond \tilde{\mathscr X_2}$ is the FIG, $\tilde{\mathscr X}=(V,E,I,\varepsilon,\rho,\eta)$ such that: $V=V_1 \times V_2$, $E=\{(a_1,b_1)(a_2,b_2)| \quad a_1a_2\in E_1 \text{ and } b_1b_2\in E_2\}$ and $I=\{((a_1,b_1), (a_1,b_1)(a_2,b_2))| \quad (a_1,a_1a_2), (a_2,a_1a_2)\in I_1 \text{ and } (b_1,b_1b_2), (b_2,b_1b_2)\in I_2\}$ and 
$$\varepsilon(a_1, b_1)= \varepsilon_1(a_1)\wedge\varepsilon_2(b_1) \quad \forall (a_1,a_2)\in V_1\times V_2
$$
$$\rho((a_1,b_1)(a_2,b_2))=\rho_1(a_1a_2)\wedge\rho_2(b_1b_2) 
$$
$$\eta((a_1,b_1), (a_1,b_1)(a_2,b_2))= \eta_1(a_1,a_1a_2)\wedge\eta_2(b_1,b_1b_2)
$$
\end{defn}
\begin{thm}\label{30}
Let $\tilde{\mathscr X_1}$ and $\tilde{\mathscr X_2}$ be two FIGs with effective pairs, then $\tilde{\mathscr X}=\tilde{\mathscr X_1}\diamond\tilde{\mathscr X_2}$ is a FIG with effective pairs.
\end{thm}
\begin{proof}
Let $\tilde{\mathscr X}=\tilde{\mathscr X_1}\diamond\tilde{\mathscr X_2}$ be the tensor product of FIGs, $\tilde{\mathscr X_1}$ and $\tilde{\mathscr X_2}$, with effective pairs.
Let $((a_1,b_1), (a_1,b_1)(a_2,b_2))$ be an arbitrary pair in $\tilde{\mathscr X}$.\\
Then by Definition \ref{29}, 
\begin{align*}
    \begin{split}
        \eta((a_1,b_1), (a_1,&b_1)(a_2,b_2))= \eta_1(a_1,a_1a_2)\wedge\eta_2(b_1,b_1b_2)\\
        &=\varepsilon_1(a_1)\wedge\rho_1(a_1a_2)\wedge\varepsilon_2(b_1)\wedge\rho_2(b_1b_2)\\
        &=\varepsilon_1(a_1)\wedge\varepsilon_2(b_1)\wedge\rho_1(a_1a_2)\wedge\rho_2(b_1b_2)\\
        &=\varepsilon(a_1,b_1)\wedge\rho((a_1,b_1)(a_2,b_2))
    \end{split}
\end{align*}
This implies that $((a_1,b_1), (a_1,b_1)(a_2,b_2))$ is an effective pair and since $((a_1,b_1), (a_1,b_1)(a_2,b_2))$ is an arbitrary pair in $\tilde{\mathscr X}$, every pair in $\tilde{\mathscr X}$ is an effective pair.
\end{proof}
\begin{rem}\label{31}
The tensor product of two CFIGs need not be a CFIG as in Example \ref{32}.
\end{rem}
\begin{exam}\label{32}
Consider Fig. 4. The FIG, $\tilde{\mathscr X}$ is the tensor product of FIGs, $\tilde{\mathscr X_1} $ and $\tilde{\mathscr X_2}$. It is clear that $\tilde{\mathscr X_1} $ and $\tilde{\mathscr X_2}$ are CFIGs but, $\tilde{\mathscr X}$ is not a CFIG.
\end{exam}
\begin{figure}
    \centering
    \includegraphics[width=7cm,height=9cm]{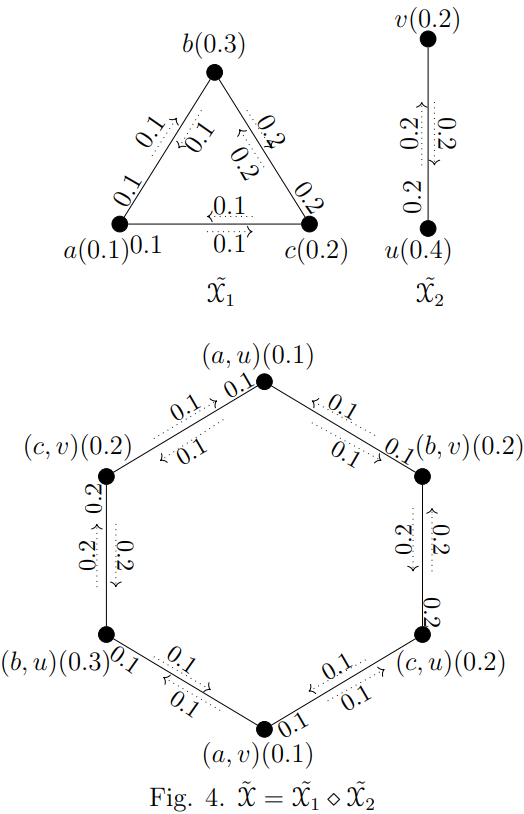}
\end{figure}
\begin{thm}\label{33}
Let $\tilde{\mathscr X_1}$ and $\tilde{\mathscr X_2}$ be two SFIGs, then $\tilde{\mathscr X}=\tilde{\mathscr X_1}\diamond\tilde{\mathscr X_2}$ is a SFIG. 
\end{thm}
\begin{proof}
Let $\tilde{\mathscr X_1}$ and $\tilde{\mathscr X_2}$ be two SFIGs, and $\tilde{\mathscr X}=\tilde{\mathscr X_1}\diamond\tilde{\mathscr X_2}$ be the tensor product of $\tilde{\mathscr X_1}$ and $\tilde{\mathscr X_2}$. Let \\ $((a_1,b_1), (a_1,b_1)(a_2,b_2))$ be an arbitrary pair in $\tilde{\mathscr X}$. Corresponding to each pair in $\tilde{\mathscr X}$, there exists one pair in each of $\tilde{\mathscr X_1}$ and $\tilde{\mathscr X_2}$, i.e., for $((a_1,b_1), (a_1,b_1)(a_2,b_2))\in \eta^*$ , there exist $(a_1,a_1a_2)\in \eta_{1}^*$ and $(b_1,b_1b_2)\in \eta_{2}^*$ such that $\eta((a_1,b_1), (a_1,b_1)(a_2,b_2))= \eta_1(a_1,a_1a_2)\wedge\eta_2(b_1,b_1b_2)
$. Therefore corresponding to a path, say $P$ in $\tilde{\mathscr X}$ there exist paths $P_1$ and $P_2$ in $\tilde{\mathscr X_1}$ and $\tilde{\mathscr X_2}$ respectively, incidence strengths of each of which is greater than or equal to the incidence strength of $P$ in $\tilde{\mathscr X}$. Also, corresponding to paths $P_1$ in $\tilde{\mathscr X_1}$ and $P_2$ in $\tilde{\mathscr X_2}$ there exists a path $P$ in $\tilde{\mathscr X}$. And by Definition \ref{29}, incidence strengths of $P_1, P_2$ is greater than or equal to the incidence strength of $P$ in $\tilde{\mathscr X}$. Therefore $ICONN_{\tilde{\mathscr{X}_1}}(a_1,a_1a_2)\wedge ICONN_{\tilde{\mathscr{X}_2}}(b_1,b_1b_2)\geq ICONN_{\tilde{\mathscr{X}}}((a_1,b_1),(a_1,b_1)(a_2,b_2))$. Now, $((a_1,b_1), (a_1,b_1)(a_2,b_2))$ is strong if 
\begin{align*}
 \eta&((a_1,b_1), (a_1,b_1)(a_2,b_2))\geq\\ & ICONN_{\tilde{\mathscr X}\setminus((a_1,b_1), (a_1,b_1)(a_2,b_2))}((a_1,b_1), (a_1,b_1)(a_2,b_2)).   
\end{align*}
Since $\tilde{\mathscr X_1}$ and $\tilde{\mathscr X_2}$ are SFIGs, every pair in $\tilde{\mathscr X_1}$ and $\tilde{\mathscr X_2}$ is strong. Hence,  
\begin{align*}
    \begin{split}
       \eta&((a_1,b_1), (a_1,b_1)(a_2,b_2))= \eta_1(a_1,a_1a_2)\wedge\eta_2(b_1,b_1b_2)\\
       &=ICONN_{\tilde{\mathscr{X}_1}}(a_1,a_1a_2)\wedge ICONN_{\tilde{\mathscr{X}_2}}(b_1,b_1b_2) \\
       &\geq ICONN_{\tilde{\mathscr{X}}}((a_1,b_1),(a_1,b_1)(a_2,b_2))\\
       &\geq ICONN_{\tilde{\mathscr{X}}\setminus {((a_1,b_1),(a_1,b_1)(a_2,b_2))}}((a_1,b_1),(a_1,b_1)(a_2,b_2))
       \end{split}
\end{align*}
which implies that $((a_1,b_1), (a_1,b_1)(a_2,b_2))$ is strong. Since $((a_1,b_1), (a_1,b_1)(a_2,b_2))$ is arbitrary, every pair in $\tilde{\mathscr X}$ is strong, and therefore $\tilde{\mathscr X}$ is a SFIG.
\end{proof}
\begin{thm}\label{34}
Let $\tilde{\mathscr X_1}$ and $\tilde{\mathscr X_2}$ be two SFIGs without isolated vertices, and $D_1$ and $D_2$ are SIDS of $\tilde{\mathscr X_1}$ and $\tilde{\mathscr X_2}$ respectively. Then $D_1\times V2$ and $V_1\times D_2 $ are SIDS of $\tilde{\mathscr X}$ and $$\gamma_{IS}\leq \wedge\{W(D_1 \times V_2), W(V_1\times D_2)\}.$$
\end{thm}
\begin{proof}
Let $\tilde{\mathscr X_1}$ and $\tilde{\mathscr X_2}$ be two SFIGs without isolated vertices, and $D_1$ and $D_2$ are  SIDS of $\tilde{\mathscr X_1}$ and $\tilde{\mathscr X_2}$ respectively.\\
Consider $D_1\times V2$, and if $(u,v)\notin D_1\times V2$. then $u\notin D_1$.
Since $D_1$ is a SIDS of $\tilde{\mathscr X_1}$, there exists $u'\in D_1$ such that $(u,uu')$ and $(u',uu')$ are strong pairs in $\tilde{\mathscr X_1}$, i.e., $u'$ dominates $u$ in $\tilde{\mathscr X_1}$. Also, as $\tilde{\mathscr X_1}$ and $\tilde{\mathscr X_2}$ are SFIGs without isolated vertices, there exists $v'\in V_2$ such that $v\neq v'$ and $vv'\in E_2$, $v $ dominates $v'$ in $\tilde{\mathscr X_2}$.\\
Therefore, $(u,v)(u',v')\in E$ and $((u,v),(u,v)(u',v'))$ and $((u',v'),(u,v)(u',v'))$ are strong pairs in $\tilde{\mathscr X}$, i.e., $(u',v')$ dominates $(u,v)$.
Hence $D_1 \times V_2$ is a SIDS in $\tilde{\mathscr X}$. Similarly, $V_1\times D_2$ is also a SIDS in $\tilde{\mathscr X}$.\\
Therefore, clearly $\gamma_{IS}\leq \wedge\{W(D_1 \times V_2), W(V_1\times D_2)\}$.
\end{proof}

\section{Strong Incidence Domination in Composition of Fuzzy Incidence Graphs}
The section studies the composition of two FIGs. Example \ref{37} illustrates that in general the composition of two SFIG need not be a SFIG. Theorem \ref{38} proves a sufficient condition for the composition of two SFIGs to be a SFIG. Theorem \ref{39} obtains a bound for the SIDN in the composition. Proposition \ref{41} deals with the composition of CFIGs.
\begin{defn}\label{35}
Let $\tilde{\mathscr X_1}=(V_1,E_1,I_1,\varepsilon_1,\rho_1,\eta_1)$ and $\tilde{\mathscr X_2}=(V_2,E_2,I_2,\varepsilon_2,\rho_2,\eta_2)$ be two FIGs. Then the composition of $\tilde{\mathscr X_1}$ and $\tilde{\mathscr X_2}$ denoted as $\tilde{\mathscr X_1}[\tilde{\mathscr X_2}]$ is the FIG, $\tilde{\mathscr X}=(V,E,I,\varepsilon,\rho,\eta)$ such that: $V=V_1 \times V_2$, $E=\{(a_1,b_1)(a_2,b_2)| a_1=a_2,\, b_1b_2 \in E_2 \text{ \, or \, } a_1a_2\in E_1\}$ and $I=\{((a_1,b_1), (a_1,b_1)(a_2,b_2))| a_1=a_2, (b_1,b_1b_2)\in I_2, (b_2,b_1b_2)\in I_2 \text{ \,or\, }(a_1,a_1a_2)\in I_1, (a_2,a_1a_2)\in I_1\}$ and 
\scriptsize{
$$\varepsilon(a_1, a_2)= \varepsilon_1(a_1)\wedge\varepsilon_2(a_2) \quad \forall (a_1,a_2)\in V_1\times V_2
$$
$$\rho((a_1,b_1)(a_2,b_2))=
\begin{cases}
\varepsilon_1(a_1)\wedge\rho_2(b_1b_2)  & if\; a_1=a_2 , b_1b_2\in E_2\\
\rho_1(a_1a_2)\wedge\varepsilon_2(b_1)  & if\; b_1=b_2 , a_1a_2\in E_1\\
\rho_1(a_1a_2)\wedge\varepsilon_2(b_1)\wedge\varepsilon_2(b_2)  & if\; b_1\neq b_2 , a_1a_2\in E_1
\end{cases} 
$$
$$\eta((a_1,b_1), (a_1,b_1)(a_2,b_2))=
\begin{cases}
\varepsilon_1(a_1)\wedge\eta_2(b_1,b_1b_2)  & if\; a_1=a_2 , \\ & \qquad (b_1,b_1b_2)\in I_2\\
\eta_1(a_1,a_1a_2)\wedge\varepsilon_2(b_1)  & if\; b_1=b_2 , \\ & \qquad (a_1,a_1a_2)\in I_1\\
\eta_1(a_1,a_1a_2)\wedge\varepsilon_2(b_1)\\ \wedge\varepsilon_2(b_2)  & if\; b_1\neq b_2 , \\ & \qquad (a_1,a_1a_2)\in I_1
\end{cases}
$$
}
\end{defn}
\begin{rem}\label{36}
In general the composition of two SFIGs need not be a SFIG. Example \ref{37} is the illustration for the same.
\end{rem}
\begin{figure}
    \centering
    \includegraphics[width=7cm, height=10cm]{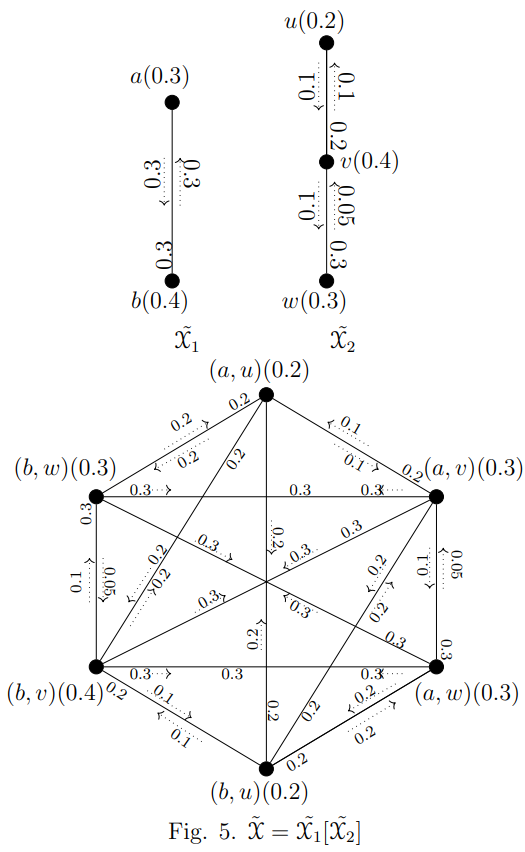}
\end{figure}
\begin{exam}\label{37}
Consider Fig. 5. Consider the cycle $C:(a,v)(b,v)(a,w)(a,v)$ in $\tilde{\mathscr X}$. The cycle $C$ has only one weakest pair $((a,w),(a,w)(a,v))$. Therefore by Theorem \ref{4}, $\tilde{\mathscr X}$ is not SFIG. 
\end{exam}

A sufficient condition for the composition of two SFIGs to be a SFIG is proved in Theorem \ref{38}.

\begin{thm}\label{38}
Let $\tilde{\mathscr X_1}$ and $\tilde{\mathscr X_2}$ be two SFIGs such that the maximum of weight of pairs in $\tilde{\mathscr X_1}$ is less than or equal to the minimum of weight of pairs in $\tilde{\mathscr X_2}$. Then the composition $\tilde{\mathscr X_1}[\tilde{\mathscr X_2}]$ is a SFIG. 
\end{thm}
\begin{proof}
Let $\tilde{\mathscr X_1}$ and $\tilde{\mathscr X_2}$ be two SFIGs such that the maximum of weight of pairs in $\tilde{\mathscr X_1}$ is less than or equal to the minimum of weight of pairs in $\tilde{\mathscr X_2}$. Let $\tilde{\mathscr X_1}[\tilde{\mathscr X_2}]$ be the composition of $\tilde{\mathscr X_1}$ and $\tilde{\mathscr X_2}$. Consider a pair in $\tilde{\mathscr X_1}[\tilde{\mathscr X_2}]$, say $((a_1,b_1),(a_1,b_1)(a_2,b_2))$. Then, either $a_1a_2\in E_1$ or $a_1=a_2, \, b_1b_2\in E_2$. The pair is strong if $\eta((a_1,b_1),(a_1,b_1)(a_2,b_2))\geq ICONN_{\tilde{\mathscr X}\setminus((a_1,b_1),(a_1,b_1)(a_2,b_2))}((a_1,b_1),(a_1,b_1)(a_2,b_2))$. There are two cases;\\
\textbf{Case 1}: $a_1a_2\in E_1$.\\
Since the maximum of weight of pairs in $\tilde{\mathscr X_1}$ is less than or equal to the minimum of weight of pairs in $\tilde{\mathscr X_2}$, in this case, weight of the pair $((a_1,b_1),(a_1,b_1)(a_2,b_2))$ is
\begin{align*}
 \eta((a_1,b_1),(a_1,b_1)(a_2,b_2))&=\eta_1(a_1,a_1a_2)\wedge\varepsilon_2(b_1)\wedge\varepsilon_2(b_2)\\
 &=\eta_1(a_1,a_1a_2)
\end{align*}
Now, consider any path from $(a_1,b_1)$ to $(a_2,b_2)$ that consists of pairs of the form $((a_1,b_n),(a_1,b_n)(a_2,b_m))$, $b_n,b_m\in V_2$. The weight of such a pair is 
\begin{align*}
 \eta((a_1,b_n),(a_1,b_n)(a_2,b_m))&=\eta_1(a_1,a_1a_2)\wedge\varepsilon_2(b_n)\wedge\varepsilon_2(b_m)\\
 &=\eta_1(a_1,a_1a_2)
\end{align*}
Hence the strength of any such path is $\leq \eta_1(a_1,a_1a_2)$.\\
Now, consider a path from $(a_1,b_1)$ to $(a_2,b_2)$ that consists of vertices of the form $(a_m,b_n)$, $a_m\in V_1,$ $a_m\neq a_1,a_2$ and $b_n\in V_2$. Then corresponding to any such path say $P_1: (a_1,b_1),(u_1,v_1),...,(u_k,v_k),(a_2,b_2)$ in $\tilde{\mathscr X}$, $u_1,u_2,...,u_k\in \tilde{\mathscr X_1} $, and $v_1,v_2,...,v_k\in \tilde{\mathscr X_2} $, there exists a walk $P_2: a_1,u_1,...,u_k,a_2$ in $\tilde{\mathscr X_1}$. This walk together with $\{(a_2,a_1a_2),  a_1a_2,(a_1,a_1a_2)\}$ consists of a cycle. Since $\tilde{\mathscr X_1}$ is a SFIG there exists a pair of weight less than or equal to $ \eta_1(a_1,a_1a_2)$. Hence, the strength of path $P_1\cup \{(a_1,b_1)(a_2,b_2),((a_2,b_2),(a_1,b_1)(a_2,b_2))\}$ is also less than or equal to $ \eta_1(a_1,a_1a_2)$. 
Therefore the pair $((a_1,b_1),(a_1,b_1)(a_2,b_2))$ is strong.\\
\textbf{Case 2}: $a_1=a_2, \, b_1b_2\in E_2$.
There are two sub cases;\\
\textbf{Sub case 1}: $\eta((a_1,b_1),(a_1,b_1)(a_1,b_2))=\varepsilon_1(a_1)$.
Now, consider a path from $(a_1,b_1)$ to $(a_1,b_2)$ that consists of vertices of the form $(a_1,b_n)$, $b_n\in V_2$. Let one such path be $P_2: (a_1,b_1),(a_1,v_1),(a_1,v_2),...,(a_1,v_k),(a_1,b_2)$, $v_1,v_2,...,v_k\in V_2$. The weight of the pairs in the path $P_2$ are \\
\begin{equation}\tag{1}
 \begin{array}{lcl}
    \eta((a_1,b_1),(a_1,b_1)(a_1,v_1))&=&\varepsilon_1(a_1)\wedge\eta_2(b_1,b_1v_1)\\
    \eta((a_1,v_1),(a_1,b_1)(a_1,v_1))&=&\varepsilon_1(a_1)\wedge\eta_2(v_1,b_1v_1)\\
    \eta((a_1,v_1),(a_1,v_1)(a_1,v_2))&=&\varepsilon_1(a_1)\wedge\eta_2(v_1,v_1v_2)\\
    \eta((a_1,v_2),(a_1,v_1)(a_1,v_2))&=&\varepsilon_1(a_1)\wedge\eta_2(v_2,v_1v_2)\\
    \vdots\\
    \eta((a_1,v_k),(a_1,v_k)(a_1,b_2))&=&\varepsilon_1(a_1)\wedge\eta_2(v_k,v_kb_2)\\
    \eta((a_1,b_2),(a_1,v_k)(a_1,b_2))&=&\varepsilon_1(a_1)\wedge\eta_2(b_2,v_kb_2)\\
 \end{array}
\end{equation}
Hence the strength of the path is less than or equal to $ \varepsilon_1(a_1)$.\\
Now, consider a path that consists of vertices of the form $(a_m,b_n)$, $b_n\in V_2$ and $a_m\neq a_1$. Then in any such path there exists a pair of the form $((a_1,b_n),(a_1,b_n)(a_m,v_k))$, $v_k\in V_2$ and the weight of the pair is 
\begin{align*}
    \eta((a_1,b_n),(a_1,b_n)(a_m,v_k))&=\eta_1(a_1,a_1a_m)\wedge\varepsilon_2(b_n)\wedge\varepsilon_2(v_k)\\
    &=\eta_1(a_1,a_1a_m)\\
    &\leq \varepsilon_1(a_1)
\end{align*}
Hence, the strength of each such path is less than or equal to $ \varepsilon_1(a_1)$. Therefore, the pair $((a_1,b_1),(a_1,b_1)(a_2,b_2))$ is strong.
\\
\textbf{Sub case 2}: $\eta((a_1,b_1),(a_1,b_1)(a_1,b_2))=\eta_2(b_1,b_1b_2)$.\\
As sub case 1, consider a path say, $P_1:(a_1,b_1),(a_1,v_1),(a_1,v_2),...,(a_1,v_k),(a_1,b_2)$, $v_1,v_2,...,v_k\in V_2$ that consists of vertices of the form $(a_1,b_n)$, $b_n\in V_2$. Then corresponding to each such path there exists a path $P_2:b_1,v_1,v_2,...,v_k,b_2$ in $\tilde{\mathscr X_2}$. Then $P_2$ along with $\{(b_2,b_1b_2), b_1b_2, (b_1,b_1b_2)\}$ forms a cycle. Since $\tilde{\mathscr X_2}$ is SFIG, there exists a pair of weight $\leq \eta_2(b_1,b_1b_2)$. Hence by (1), the path $P_1\cup\{(a_1,b_1)(a_2,b_2),((a_2,b_2),(a_1,b_1)(a_2,b_2)) \}$ has strength less than or equal to $ \eta_2(b_1,b_1b_2)$.
Now, consider a path that consists of vertices of the form $(a_m,b_n)$, $b_n\in V_2$ and $a_m\neq a_1$. Then in each such path there exists a pair of the form $((a_1,b_n),(a_1,b_n)(a_m,v_k))$, $v_k\in V_2$ and since the maximum of weight of pairs in $\tilde{\mathscr X_1}$ is less than or equal to the minimum of weight of pairs in $\tilde{\mathscr X_2}$,
\begin{align*}
    \eta((a_1,b_n),(a_1,b_n)(a_m,v_k))&=\eta_1(a_1,a_1a_m)\wedge\varepsilon_2(b_n)\wedge\varepsilon_2(v_k)\\
    &=\eta_1(a_1,a_1a_m)\\
    &\leq \eta_2(b_1,b_1b_2).
\end{align*}
Hence every path from $(a_1, b_1) $ to $(a_2,b_2)$ has strength $\leq \eta_2(b_1,b_1b_2)$. Therefore, the pair $((a_1,b_1),(a_1,b_1)(a_2,b_2))$ is strong.
\end{proof}
\begin{thm}\label{39}
Let $\tilde{\mathscr X_1}$ and $\tilde{\mathscr X_2}$ be FIGs such that $\tilde{\mathscr X}=\tilde{\mathscr X_1}[\tilde{\mathscr X_2}]$ is SFIG. Let $D_1$ and $D_2$ be SIDSs of $\tilde{\mathscr X_1}$ and $\tilde{\mathscr X_2}$ respectively. Then $D_1\times D_2$ is a SIDS of $\tilde{\mathscr X}$ and $\gamma_{IS}\leq W(D_1\times D_2)$.
\end{thm}
\begin{proof}
Let $\tilde{\mathscr X_1}$ and $\tilde{\mathscr X_2}$ be FIGs such that $\tilde{\mathscr X}=\tilde{\mathscr X_1}[\tilde{\mathscr X_2}]$ is a SFIG. Let $D_1$ and $D_2$ be SIDS of $\tilde{\mathscr X_1}$ and $\tilde{\mathscr X_2}$ respectively.
Consider $(u,v)\notin D_1\times D_2$. Then there are three cases;\\
\textbf{Case 1}: $u\notin D_1$ and $v\in D_2$.\\
Since $D_1$ is a SIDS of $\tilde{\mathscr X_1}$, there exists a $u'\in D_1$ that dominates $u$. By Definition \ref{35}, $(u',v)(u,v)\in E$. Since $\tilde{\mathscr X}$ is SFIG, $(u',v)$ dominates $(u,v)$. \\
\textbf{Case 2}: $u\in D_1$ and $v\notin D_2$.\\
Since $D_2$ is a SIDS of $\tilde{\mathscr X_2}$, there exists a $v'\in D_2$ that dominates $v$. By Definition \ref{35}, $(u,v)(u,v')\in E$. Since $\tilde{\mathscr X}$ is SFIG, $(u,v')$ dominates $(u,v)$. \\
\textbf{Case 3}: $u\notin D_1$ and $v\notin D_2$.\\
Since $D_1$ is a SIDS of $\tilde{\mathscr X_1}$, there exists a $u'\in D_1$ that dominates $u$. By Definition \ref{35}, $(u',v)(u,v)\in E$. Since $\tilde{\mathscr X}$ is SFIG, $(u',v)$ dominates $(u,v)$. \\
Therefore, $D_1\times D_2$ is a SIDS of $\tilde{\mathscr X}$ and hence, $\gamma_{IS}\leq W(D_1\times D_2)$
\end{proof}
\begin{cor}\label{40}
Let $\tilde{\mathscr X_1}$ and $\tilde{\mathscr X_2}$ be two SFIGs such that the maximum of weight of pairs in $\tilde{\mathscr X_1}$ is less than or equal to the minimum of weight of pairs in $\tilde{\mathscr X_2}$. Then $D_1\times D_2$ is a SIDS of $\tilde{\mathscr X}$ and $\gamma_{IS}\leq W(D_1\times D_2)$, where $D_1$ and $D_2$ are SIDSs of $\tilde{\mathscr X_1}$ and $\tilde{\mathscr X_2}$ respectively.
\end{cor}
\begin{prop}\label{41}
Let $\tilde{\mathscr X_1}$ and $\tilde{\mathscr X_2}$ be two CFIGs, then the composition $\tilde{\mathscr X}=\tilde{\mathscr X_1}[\tilde{\mathscr X_2}]$ is a CFIG.
\end{prop}
\begin{proof}
Consider $\tilde{\mathscr X}=\tilde{\mathscr X_1}[\tilde{\mathscr X_2}]$, where $\tilde{\mathscr X_1}$ and $\tilde{\mathscr X_2}$ are two CFIGs. 
Consider any two distinct vertices $(a,u)$ and $(b,v)$ in $\tilde{\mathscr X}$. There are two cases;\\
\textbf{Case 1}: $a=b$,
Then $u\neq v$. Since $\tilde{\mathscr X_2}$ is CFIG, it implies that $uv\in E_2$. Therefore by Definition \ref{35}, $(a,u)(b,v)\in E$.\\
\textbf{Case 2}: $a\neq b$,
In this case, since $\tilde{\mathscr X_1}$ is CFIG, $ab\in E_1$. Therefore by Definition \ref{35}, $(a,u)(b,v)\in E$. \\
Hence by Case 1 and 2, the underlying graph of $\tilde{\mathscr X}$ is complete. 
Now, consider the edge $(a,u)(b,v)$ in $\tilde{\mathscr X}$.
Then weight of $(a,u)(b,v)$ is;
{\scriptsize
\begin{align*}
\rho((a,u)(b,v))&=
\begin{cases}
\varepsilon_1(a)\wedge\rho_2(uv) \quad & if \quad a=b\\
\rho_1(ab)\wedge\varepsilon_2(u) \quad & if \quad a\neq b \quad and \quad  u=v\\
\rho_1(ab)\wedge\varepsilon_2(u)\wedge\varepsilon_2(v) \quad & if \quad a\neq b\quad and \quad  u\neq v
\end{cases}\\
&=
\begin{cases}
\varepsilon_1(a)\wedge\varepsilon_2(u)\wedge\varepsilon_2(v) \quad & if \quad a=b\\
\varepsilon_1(a)\wedge\varepsilon_1(b)\wedge\varepsilon_2(u) \quad & if \quad a\neq b \quad and \quad  u=v\\
\varepsilon_1(a)\wedge\varepsilon_1(b)\wedge\varepsilon_2(u)\wedge\varepsilon_2(v) \quad & if \quad a\neq b\quad and \quad  u\neq v
\end{cases}\\
&= \varepsilon(a,u)\wedge\varepsilon(b,v)
\end{align*}
}
Similarly, by Definition \ref{35} the weight of the pair $((a,u),(a,u)(b,v))$ is;
{\scriptsize
\begin{align*}
\eta((a,u),(a,u)(b,v))&=
\begin{cases}
\varepsilon_1(a)\wedge\eta_2(u,uv) \quad & if \quad a=b\\
\eta_1(a,ab)\wedge\varepsilon_2(u) \quad & if \quad a\neq b \quad and \quad  u=v\\
\eta_1(a,ab)\wedge\varepsilon_2(u)\wedge\varepsilon_2(v) \quad & if \quad a\neq b\quad and \quad  u\neq v
\end{cases}\\
&=
\begin{cases}
\varepsilon_1(a)\wedge\varepsilon_2(u)\wedge\rho_2(uv) \quad & if \quad a=b\\
\varepsilon_1(a)\wedge\rho_1(ab)\wedge\varepsilon_2(u) \quad & if \quad a\neq b \quad and \\& \qquad  u=v\\
\varepsilon_1(a)\wedge\rho_1(ab)\wedge\varepsilon_2(u)\wedge\varepsilon_2(v) \quad & if \quad a\neq b\quad and \\ &\qquad  u\neq v
\end{cases}\\
&= \varepsilon(a,u)\wedge\rho((a,u)(b,v))
\end{align*}
}
By the same way, weight of $((b,v),(a,u)(b,v))$ is $\eta((b,v),(a,u)(b,v))=\varepsilon(b,v)\wedge\rho((b,v)(a,u))$.
Hence, $\tilde{\mathscr X}$ is a CFIG.
\end{proof}

\section{Application}
In the research world it is very common for research groups to collaborate and discuss research interests within the groups. Here, as an application of tensor product, a collaboration graph is considered.\\
Consider two research groups. Each group is represented by a FIG with vertices representing the people in the group. An edge joins two vertices $a$ and $b$ if person $a$ is involved in a research discussion with person $b$. Each vertex is assigned with a weight 1, representing their involvement in the discussion. The edge $ab$ represent the discussion between $a$ and $b$ and edge weight $\rho(ab)$ is the total discussions between person $a$ and $b$. The pair weight represents their individual contribution to the discussion i.e. a pair $(a,ab)$ is the contribution of $a$ to the discussion between $a$ and $b$ . Now, suppose that the two groups wish to collaborate. Now a new group is formed with vertex representing a team of two people. Vertex in the new group is of the form $(a,u)$, where person $a$ is from the first group and person $u$ is from the second group. Vertices $(a,u)$ and $(b,v)$ are adjacent if $a$ and $b$ are engaged in a discussion in the first group, and $u$ and $v$ are engaged in a discussion in the second group. Now consider the tensor product of the two graphs. The edge weight is the minimum discussion between the two teams, and the pair weight is the minimum contribution of each team to the discussion.\\  
Now, consider Fig. 6. FIGs $\tilde{\mathscr{X}_1}$ and $\tilde{\mathscr{X}_2}$ represents two research groups. Consider the FIG, $\tilde{\mathscr{X}}$ with collaboration graph as underlying graph. Each vertex represent two people, one from each group. Discussion happens between two teams $(a,u)$ and $(b,v)$ of the group if $a$ and $b$ are involved in a discussion in the first group and, $u$ and $v$ are involved in a discussion in the second group. It can be observed that the FIG $\tilde{\mathscr{X}}$ is the tensor product of FIGs $\tilde{\mathscr{X}_1}$ and $\tilde{\mathscr{X}_2}$. Here edge weight $\rho((a,u)(b,v))$ is the minimum of discussion between $a$ and $b$, and $u$ and $v$. Similarly, the pair weight $\eta((a,u),(a,u)(b,v))$ is the minimum of individual contribution of $a$ and $u$ to the discussion $ab$ and $uv$ respectively.  
\begin{figure}[H]
    \centering
    \includegraphics[width=8cm,height=10cm]{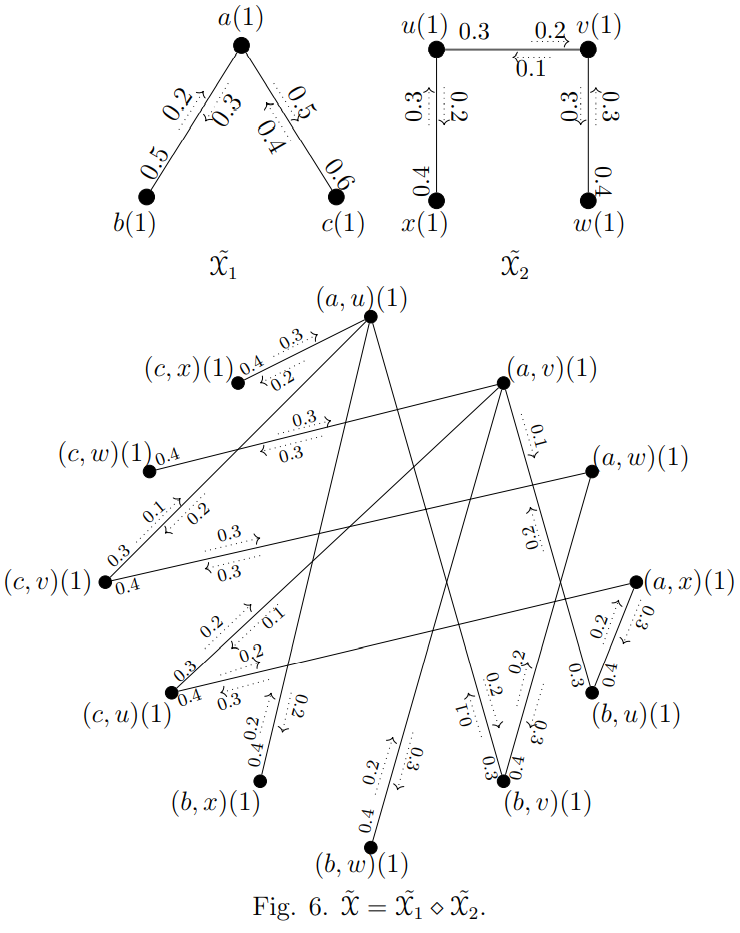}
\end{figure}
In FIG, $\tilde{\mathscr{X}}$ the concept of SID can also be considered. A minimum SIDS of  $\tilde{\mathscr{X}}$ is $\{(a,u), (a,v), (b,v), (c,u) \}$ which means 4 teams are required for the discussion to happen in the entire group. And SIDN is 0.6, which is the minimum individual contribution of the 4 teams in the SIDS to the discussion.

\section{Conclusion}
The introduction of graph operations motivated the researchers to study the properties of the new graphs obtained from the known structures. Domination is an area of graph theory of growing importance and is extensively studied by researchers. Hence it is significant to combine the ideas to study SID in the operations on FIGs. The article deals with the Cartesian product, join, tensor product, and composition of FIGs and some of the basic properties of the obtained graphs. The study mainly focus on SFIGs and FIGs with effective pairs. SID is also discussed in the FIGs obtained from the operations. Bounds for the SIDN of FIG are also obtained for each operation.
\\\\\\
\textbf{\large{Acknowledgement}}\\\\
The first author gratefully acknowledges the financial support of Council of Science and Industrial Research (CSIR), Government of India.\\
The authors would like to thank the DST, Government of India, for providing support to carry out this work under the scheme 'FIST' (No.SR/FST/MS-I/2019/40).

\end{document}